\documentclass[11pt]{amsart}
\usepackage{geometry}  
\usepackage{amsmath}
\usepackage{mathtools}
\usepackage{tikz-cd}
\usetikzlibrary{babel}             
\usepackage{eqnarray,amsmath}
\usepackage{graphicx}
\usepackage{amssymb}
\usepackage{amsthm}
\usepackage{epstopdf}
\usepackage{color}
\usepackage{tikz-cd}
\usepackage{xcolor}
\theoremstyle{plain}

\newtheorem{theorem}{Theorem}[section]

\newtheorem{corollary}[theorem]{Corollary}
\newtheorem{proposition}[theorem]{Proposition}
\newtheorem{lemma}[theorem]{Lemma}

\theoremstyle{definition}
\newtheorem{definition}[theorem]{Definition}

\newtheorem{remark}[theorem]{Remark}

\title{$V$-minimal submanifolds}
\author{Monica Alice Aprodu}
\thanks{This work was supported by Romanian Ministry of Education,
	Program PN-III, Project number PN-III-P4-ID-PCE-2020-0025, Contract 30/04.02.2021 and "Dunarea de Jos" University of Galati Grant, Project number  RF 2488/31.05.2024.}

\begin{document}
	\maketitle

	\begin{abstract}
		We introduce the notion of $V$-minimality, for $V$ a smooth vector field on a Riemannian manifold, a natural extension of the classical notion of  minimality, and we prove several basic properties. One featured example is given for locally conformal Kaehler (l.c.K) manifolds. It is well-known that in general, complex submanifolds in non-Kaehler l.c.K manifolds are not minimal. We prove that, however, they are $V$-minimal for $V$ a suitable multiple of the Lee vector field. Extending some results from \cite{AAB}, to emphasis the utility of this notion, we prove that a PHH submersion is $V$-harmonic if and only if it has minimal fibres and a PHH $V$-harmonic submersion pulls back complex submanifolds to $V$ minimal submanifolds. 
		
	\end{abstract}
	
	\section{Introduction}
	
	Let $(M,g)$ and $(N,h)$ be compact Riemannian manifolds, $V$ a smooth vector field on $M$ and $\varphi:M\to N$ a smooth map. In \cite{Zhao}, \cite{CJW} the authors introduced the notion of $V$-harmonic map, see Definition \ref{vharm}, that naturally generalizes the classical notion of harmonic map. Unlike the latter, $V$-harmonicity is not defined via a variational problem, but rather by imposing the vanishing of a modified tension field, called the $V$-tension field of $\varphi$. If $V=0$, the two notions coincide, more generally, the same is true if $V$ is vertical.
	
	In this generalized context, $V$-harmonic morphisms appear naturally, \cite{Zhao}, and are directly connected to the minimality of the fibres, exactly as in the classical case, see Theorem~\ref{thm:V-ha-minimal} in \cite{Zhao}.
	
	In the case of K\"ahler target manifolds, harmonic morphisms generalize to pseudo-harmonic morphisms \cite{Loubeau}. Furthermore, if a natural extra-condition that the naturally almost complex structure on the horizontal distribution satisfies a K\"ahler-type condition, then the harmonicity implies the minimality of the fibres. These maps, called pseudo-horizontally homothetic (PHH) enjoy other geometric properties, \cite{AAB}.
	
	In the presence of a vector field $V$ on a manifold $M$, one can naturally ask whether the definition of minimality can be perturbed to underline the influence of $V$.
	The main goal of this paper is to propose a natural definition which we call $V$-minimality, Definition \ref{def:V-minimal}. As for $V$-harmonicity, minimality corresponds to the case $V=0$. One specific example is of particular interest. It is known that a complex submanifold $K$ in a locally conformal K\"ahler manifold $M$ is not minimal, unless the Lee vector field is tangent to the subvariety, Theorem 12.1, \cite{DO}. However, we notice that $K$ is $V$-minimal, for $V$ a suitable integer multiple of the Lee vector field of $M$, see Theorem \ref{thm:lck}. Furthermore, we connect $V$-minimality to PHH maps and their $V$-harmonicity.

		The outline of the paper is as follows. Firstly, we recall the notions needed for our aim, $V$-harmonic maps, PHH and PHWC maps, l.c.K manifolds and we briefly review some of their properties. 
		We then define $V$-minimal submanifolds, Definition \ref{def:V-minimal} and $V$-pseudo harmonic morphisms, Definition \ref{def:PVHM}. The main results of the paper are Theorem \ref{thm:lck},
		Theorem \ref{tpvhm}, Theorem \ref{p13} and Theorem \ref{thm:pullback}.
		Theorem \ref{p13} shows that for a PHH submersion, $V$-harmonicity is equivalent to having $V$-minimal fibres. In Theorem \ref{thm:pullback}, we prove that a PHH $V$-harmonic submersion pulls back complex submanifolds to $V$-minimal submanifolds.

	\section{$V$ - Harmonic Morphisms and $V$ - minimal submanifolds}
	
	\subsection{$V$ - Harmonic Maps and Morphisms}
	
	We remind the reader of some basic facts on $V$-Harmonic Maps and Morphisms that we shall use in the sequel \cite{Zhao},\cite{CJW}, \cite{CJQ}, \cite{CQ}, \cite{Q}.
	
	\begin{definition}[see \cite{CJW}, \cite{Zhao}]
		\label{vharm}
		Let $(M^m,g)$ and $(N^n,h)$ be two Riemannian manifolds, $V$ a smooth vector field on $M$, and $\varphi: M\rightarrow N$ a smooth map. The map $\varphi$ is called {\it $V$-harmonic} if satisfies:
		\begin{equation}\label{v}
			\tau_V(\varphi):=\tau(\varphi)+d\varphi(V)=0,
		\end{equation}
		where $\tau(\varphi)$ is the tension field of the map $\varphi.$ Since the differential $d\varphi$ of $\varphi$ can be viewed as a section of the bundle $T^*M\otimes \varphi^{-1}TN,$ $d\varphi(V)$ is a section of the bundle $\varphi^{-1}TN.$ The tension $\tau_V(\varphi)$ is called the {\it $V$-tension field of $\varphi$}. 
	\end{definition}
	
	\begin{remark}[cf. \cite{CJW}, \cite{Zhao}] In particular,
		\begin{itemize}
			\item [1.] For $V=0,$ the map $\varphi$ is harmonic.
			\item [2.] A smooth function $f:M\rightarrow \mathbf{R}$ is said to be {\it $V$-harmonic} if: 
			$$\Delta_V(f):=\Delta (f)+\langle V,\nabla f\rangle=0.$$
		\end{itemize}
	\end{remark}
	By taking the trace of the second fundamental form, we obtain the tension field of the map $\varphi,$ $\tau(\varphi)=\mathrm{trace} \nabla d\varphi,$ which is a section of $\varphi^{-1}TN.$ In local coordinates $(x_i)_{i=\overline{1,m}}$ on $M$ and $(y_\alpha)_{\alpha=\overline{1,n}}$ on $N,$ respectively, it has the following expression (see for example \cite{BairdWood}, \cite{EL}):
	\begin{equation}\label{tau1}
		\tau(\varphi)=\sum\limits_{\alpha=1}^n\tau(\varphi)^\alpha\frac{\partial}{\partial y_\alpha},
	\end{equation}
	where, denoting by ${^M\Gamma}_{ij}^k$ and ${^NL}_{\beta \gamma}^\alpha$ the Christoffel symbols of $M$ and $N$, and $\varphi^\alpha=\varphi\circ y_\alpha,$
	\begin{equation}\label{tau2}
		\tau(\varphi)^{\alpha}=
		\sum\limits_{i,j=1}^m g^{ij}\left( \frac{\partial^2\varphi^{\alpha}}{\partial x_i \partial x_j}-
		\sum\limits_{k=1}^m {^M\Gamma}_{ij}^k \frac{\partial \varphi ^{\alpha}}{\partial x_k}+
		\sum\limits_{\beta,\gamma=1}^n {^NL}_{\beta \gamma}^\alpha\frac{\partial \varphi^\beta}{\partial x_i}\frac{\partial \varphi^\gamma}{\partial x_j}\right).
	\end{equation}
	
	For $V$ a smooth vector field on $M$, given in local coordinates on $M$ by $V=\sum\limits_{i=1}^mV_i\frac{\partial}{\partial x_i}$, the $V$-tension field of the map $\varphi$ has the following expression in local coordinates:
	\begin{equation}\label{tauv}
		\tau_V(\varphi)=
		\sum\limits_{\alpha=1}^n\tau(\varphi)^\alpha \frac{\partial}{\partial y_\alpha}+\sum\limits_{\alpha=1}^n\sum\limits_{i=1}^m V_i\frac{\partial \varphi^\alpha}{\partial x_i}\frac{\partial}{\partial y_\alpha}.
	\end{equation}
	
	Using the properties of the second fundamental form and of the tension field of the composition of two maps, (see \cite{BairdWood}), we have the following lemma:
	\begin{lemma}[see \cite{Zhao}]
		\label{comp} 
		The $V$-tension field of the composition of two maps $\varphi:M\rightarrow N$ and $\psi:N\rightarrow P$ is given by:
		\begin{equation}
			\tau_V(\psi\circ \varphi)=d\psi(\tau_V(\varphi))+\mathrm{trace} \nabla d\psi(d\varphi,d\varphi).
		\end{equation}
	\end{lemma} 
	As in the case of harmonic maps, Zhao (Definition 1.2 in \cite{Zhao}) defined {\it $V$-harmonic morphisms} as maps $\varphi:M\rightarrow N,$ between Riemannian manifolds, which pulls back local harmonic functions on $N$ to local $V$-harmonic functions on $M$.
	
	The following results characterize the $V$-harmonic morphisms.
	
	\begin{theorem}[see \cite{Zhao}]
		Let $\varphi:(M^m,g)\rightarrow (N^n,h)$ be a smooth map between Riemannian manifolds. Then the following conditions are equivalent:
		\begin{itemize}
			\item [1)] $\varphi$ is a $V$-harmonic morphism;
			\item [2)] $\varphi$ is a horizontally weakly conformal $V$-harmonic map;
			\item [3)] $\forall \psi:W\rightarrow P$ a smooth map from an open subset $W\subset N$ with $\varphi^{-1}(W)\neq\emptyset,$ to a Riemann manifold $P$, we have: $\tau_V(\psi\circ \varphi)=\lambda^2\tau(\psi),$ for some smooth function $\lambda^2:M\rightarrow [0,\infty);$
			\item [4)] $\forall \psi:W\rightarrow P$ a harmonic map from an open subset $W\subset N$ with $\varphi^{-1}(W)\neq\emptyset,$ to a Riemann manifold $P$, the map $\psi\circ \varphi$ is a $V$-harmonic map;
			\item [5)] $\exists \lambda^2:M\rightarrow [0,\infty)$ a smooth function such that: $\Delta_V(f\circ \varphi)=\lambda^2\Delta f,$ for any function $f$ defined on an open subset $W$ of $N$ with $\varphi^{-1}(W)\neq\emptyset.$
		\end{itemize}
	\end{theorem}

	\begin{corollary}[see \cite{Zhao}]
		\label{lift}
		For $\varphi:M\rightarrow N$ a $V$-harmonic morphism with dilation $\lambda$ and $\psi:N\rightarrow P$ a harmonic morphism with dilation $\theta$, the composition $\psi\circ \varphi$ is a $V$-harmonic morphism with dilation $\lambda(\theta\circ \varphi).$
	\end{corollary}

	\begin{theorem}[see \cite{Zhao}]
		\label{thm:V-ha-minimal}
		For a horizontally weakly conformal map $\varphi:M\rightarrow N$ with dilation $\lambda$, any two of the following conditions imply the third:
		\begin{itemize}
			\item [1)] $\varphi$ is a $V$-harmonic map (and so a $V$-harmonic morphism);
			\item [2)] $V+\nabla \mathrm{log}(\lambda^{2-n})$ is vertical at regular points;
			\item [3)] the fibres of $\varphi$ are minimal at regular points.
		\end{itemize}
	\end{theorem}
	\subsection{$V$ - minimal submanifolds}
	
	Let $(M,g)$ be a Riemannian manifold, $V$ a smooth vector field on $M$, and $K$ a submanifold of $M$. For any $x\in K$, we have the orthogonal decomposition of the tangent bundle $T_xM=T_xK\oplus T_xK^{\perp},$ with respect to $g_x.$ According to this decomposition, we have $${\nabla}^M_XY={\nabla}^K_XY+A(X,Y), \mbox{ for all } X,Y\in \Gamma(TK).$$ The symmetric bilinear map $A:TK\times TK\rightarrow TK^{\perp}$ is {\it the second fundamental form of the submanifold} $K$.
	
	\begin{definition}\label{def:V-minimal}
		The submanifold $K$ of $M$ is called {\it $V$-minimal} if 
		$$
		\mathrm{trace}(A)-V\in \Gamma(TK).
		$$
	\end{definition}
	\begin{remark}
		Note that, as the trace of the second fundamental form has the image in the normal bundle, the $V$-minimality can be rephrased as 
		$$
		\mathrm{trace} (A)-V^{\perp}=0.
		$$ 
	\end{remark}
	
	\begin{remark}
		Assume $\varphi:M\to N$ is a differentiable map and $K$ is a fibre of $\varphi$. Then the $V$-minimality condition translates to: $\mathrm{trace}(A)-V$ is a vertical vector. We refer to the  subsequent Proposition \ref{prop:R-Submersion} for further aspects of $V$--minimality of the fibers.
	\end{remark}

	\begin{proposition}
		Let $K$ be a closed submanifold of a Riemann manifold $(M,g).$ Then there exists $V$ a smooth vector field on $M,$ such that $K$ is $V$-minimal.
	\end{proposition}
	
	\begin{proof}
		Let $A:TK\times TK\rightarrow TK^{\perp}$ be the second fundamental form of the submanifold $K$.
		
		Consider on $K$ the vector field ${\rm trace}(A).$ It is known that it is a normal vector field on $K.$ Using a known result (see \cite{JL} Lemma 8.6 and Exercise 8-15) there exists a smooth vector field $V$ on $M,$ such that $V_{|K}={\rm trace}(A).$ This implies that $V_{|K}=V^{\perp}_{|K}$ and, moreover, ${\rm trace}(A)-V^{\perp}=0.$ In conclusion, $K$ is $V$-minimal.
	\end{proof}

	This proposition can be adapted to Riemann submersions such that the result holds true for a vector field $V$ that is universal across all fibers, as follows.
	
	Recall that (see\cite{Neill}) for a Riemannian submersion $\varphi: M\rightarrow N,$ can be defined two tensors, one of which is the second fundamental form of all the fibers.
	
	If $\mathcal{H}$ and $\mathcal{V}$ denote the horizontal and the vertical distribution on $M$, then the second fundamental form of all fibers $\varphi^{-1}(y),$ $y\in N,$ gives rise to a (1,2)-tensor field on $M$, defined by:
	$$T_XY=\mathcal{H}\nabla^M_{\mathcal{V}X}(\mathcal{V}Y)+\mathcal{V}\nabla^M_{\mathcal{V}X}(\mathcal{H}Y),
	$$ 
	for all arbitrary vector field $X,Y$ in $M$ and $\nabla^M$ the covariant derivative on $M$.
	
	It is known that the tensor $T$ has the following properties:
	\begin{itemize}
		\item [1)] $T_X$ is a skew-symmetric linear operator on $TM$ and it reverses the horizontal and vertical subspaces.
		\item [2)] $T_X=T_{\mathcal{V}X},$ for all $X$ in $M$.
		\item [3)] $T_XY=T_YX,$ for all $X,Y$ vertical vector fields on $M$.
		\item [4)] If $X,Y$ are horizontal vector fields, and $V,W$ are vertical vector fields on $M$, then: $\nabla^M_VW=T_VW+\mathcal{V}\nabla^M_VW$ and $\nabla^M _VX=\mathcal{H}\nabla^M_VX+T_VX.$
	\end{itemize}
	
	In this context, we prove the following.
	
	\begin{proposition}
		\label{prop:R-Submersion}
		Let $\varphi:M\rightarrow N$ be a Riemannian submersion. Then there exists a smooth vector field $V$ on $M$ such that any fiber of $\varphi$ is $V$-minimal.
	\end{proposition}
	
	\begin{proof}
		Let us consider $T_XY=\mathcal{H}\nabla^M_{\mathcal{V}X}(\mathcal{V}Y)+\mathcal{V}\nabla^M_{\mathcal{V}X}(\mathcal{H}Y),$ the (1,2)-tensor field on $M$, described above.
		
		Choose the vector field $V$ on $M$, to be the $\rm{trace} (\mathcal{H}\nabla^M_{\mathcal{V}X}(\mathcal{V}Y))$.
		
		For any $y\in N$, denote by $K$ the fiber $\varphi^{-1}(y).$ Then, $$V_{|K}=\rm{trace}(\mathcal{H}\nabla^M_{X}Y), \forall X,Y\in \Gamma(TK).$$
		But, the horizontal component of $\nabla^M_XY, \forall X,Y\in \Gamma(TK)$ is exactly the second fundamental form of $K$ in $M$, so is equal to $A(X,Y).$
		
		Thus, $\rm{trace}(A)-V^\perp_{|K}=\rm{trace}(A)-V_{|K}=0,$ so the fiber $K$ is $V$-minimal.
	\end{proof}

	This propositions provide a plethora of examples. However, the most interesting ones are obtained when $V$ is a naturally defined vector field on $M$, as exemplified in the next subsection. 
	
	\subsection{$V$--minimality and locally conformal K\"ahler manifolds}	
	We recall here briefly the definitions of a locally conformal K\"ahler manifold and a locally conformal K\"ahler submersion.
	
	\begin{definition}(\cite{DO}, \cite{Vaisman})
		A complex n-dimensional Hermitian manifold $(M^{2n}, J, g)$, where $J$
		denotes its complex structure and $g$ its Hermitian metric is called
		a {\it locally conformal K\" ahler (l.c.K.) manifold} if there is an open cover $\{U_i\}_{i\in I}$ of $M$ and a family of $\mathcal{C}^\infty$ functions $f_i:U_i\rightarrow \mathbf{R}$, $i\in I,$ such that each local
		metric: 
		\begin{equation}\label{lck}
			g_i=\mathrm{exp}(-f_i)g_{|U_i}
		\end{equation} is K\" ahlerian, where $g_{|U_i}=\iota^\star_i,$ and $\iota_i:U_i\rightarrow M$ is the inclusion.
		
		The manifold is called {\it globally conformal K\"ahler (g.c.K) manifold} if there exist a $\mathcal{C}^\infty$ function $f:M\rightarrow \mathbf{R},$ such that the metric $exp(-f)g$ is K\"ahler.
	\end{definition}
	
	Let $\Omega, \Omega_i$ be the 2-forms associated with $(J, g)$ and $(J, g_i),$ respectively ($\Omega (\cdot, \cdot) = g(\cdot, J\cdot)$).
	Then (\ref{lck}) implies: 
	\begin{equation}
		\Omega_i=\mathrm{exp}(-f_i)\Omega_{|U_i}.
	\end{equation}
	
	A Hermitian manifold $(M^{2n}, J, g)$ is l.c.K manifold if and only if there exists a globally defined cosed 1-form $\omega$ on $M,$ such that $d\Omega=\omega\wedge \Omega$ (see \cite{DO}, Theorem 1.1). The closed 1-form $\omega$ is called the {\it Lee form of the manifold}. When the manifold is g.c.K, the Lee form is exact and when the manifold is K\"ahler, the Lee form is zero.
	
	The vector field $B=\omega^\sharp$ on $M$, described by the condition by $\omega(X)=g(X,B),$ for all $X$ in $M$, is called the {\it Lee vector field} of $M$.
	
	\medskip
	
	A complex submanifolds of a l.c.K. manifold is seldom minimal, see Theorem 12.1, \cite{DO}. However, it is $V$-minimal for a suitable vector field $V$, 
	as shown below.
	
	\begin{theorem}
		\label{thm:lck}
		With the notation as above, let $K$ be a complex submanifold of a l.c.K manifold $M$ of complex dimension $m,$ and $V$ a smooth vector field on $M$. 
		Then $K$ is a $V$-minimal submanifold on $M$, for $V=-mB$.
	\end{theorem}
	
	\begin{proof}	
		Denote as usual by $A$ the second fundamental form of the complex submanifold $K\subset M.$ 
		Using eq.(12.26), Theorem 12.1, \cite{DO}, and choosing an orthonormal frame on $K$, $\{E_i,JE_i\}_{i=\overline{1,n}}$, $m={\rm{dim}_{\mathbb{C}}}K $,
		\begin{equation}
			\begin{array}{lll}
				\mathrm{trace} (A)&=&\sum_{i=1}^{m}(A(E_i,E_i)+A(JE_i,JE_i))\\
				\\
				&=&-\sum_{i=1}^{m}g(E_i,E_i)B^\perp\\
				\\
				&=&-n B^\perp.
			\end{array}
		\end{equation}
		If we choose the vector field on $M$ to be $V=-mB$, then
		$\mathrm{trace}(A)-V=-m(B^\perp-B)\in \Gamma(TK)$, which meas that $K\subset M$ is a $V=-mB$ minimal submanifold.
	\end{proof}	
	
	\begin{definition}(\cite{DO}, \cite{Marrero})
		Let $(M,J,g)$ and $(M^\prime, J^\prime, g^\prime)$ be two almost Hermitian manifolds and $\varphi:M\rightarrow M^\prime$ a Riemannian submersion, which is holomorphic (i.e. $d\varphi \circ J=J^\prime\circ d\varphi$). Moreover, if $(M,J,g)$ is a l.c.K manifold, then $\varphi$ is called a {\it l.c.K submersion}.
	\end{definition}
	
	Another example is obtained when considering l.c.K submersions, as proven in the following result.
	
	\begin{proposition}
		Let $\varphi:M\rightarrow M^\prime$ a l.c.K submersion, and $V$ a vector field on $M$, where $M$ is compact. Then, the fibers of $\varphi$ are $V$- minimal if and only if the Lee vector field $B$ of $M$ satisfies $V+mB$ is a vertical vector field, where $\rm {m}$ is the dimension of the fiber.
		
	\end{proposition}
	\begin{proof}
		Let $y\in M^\prime$, and denote by $K,$ the fiber $\varphi^{-1}(y)$ which is an immersed complex submanifold of $M$ ($K\xhookrightarrow{i} M$).
		
		The l.c.K metric $g$ on $M$ induces a l.c.K metric $g_K$ on $K$ with the Lee form $\omega_K=i^\star \omega,$ and the Lee vector field $B_K$ given by the decomposition $B=B_K+B^\perp,$ $B_K\in \Gamma(TK),$ $B^\perp\in \Gamma(TK^\perp).$
		
		Using Gauss-Weingarten equation of the submanifold $K$ we have:
		$$\nabla^M_XB_K=\nabla^K_XB_K+A(X,B_K),$$ 
		$$\nabla^M_XB^\perp=-\mathcal{A}_{B^\perp}(X)+D_XB^\perp,$$
		for all $X\in\Gamma(TK),$ where $A$ is the second fundamental form of $K$ in $M,$ $D$ is the induced connection of the normal bundle of $K,$ and $-\mathcal{A}_{B^\perp}(X)$ is the tangential component of $\nabla^M_XB^\perp.$ Moreover, for any $X,Y\in\Gamma(TK),$ 
		$$g(\mathcal{A}_{B^\perp}(X),Y)=g(A(X,Y),B^\perp),$$ and
		the mean curvature vector of $K$ is \begin{equation}\label{e1}
			\eta=\frac{1}{2m}\rm{trace}(A).
		\end{equation}
		
		If $K$ is $V$-minimal, then $\rm{trace}(A)-V$ is a vertical vector field.
		
		Keeping in mind that $\rm{trace}(A)\in \Gamma(TK^\perp),$ and decompose $V=V^v+V^\perp$ into tangent and normal component with respect to $K,$ we have 
		\begin{equation}\label{e2}
			\rm{trace}(A)-V^\perp=0.
		\end{equation}
		As $M$ is a l.c.K manifold, let us denote by $\tilde{A}$ the (local) second fundamental form of $K$ with respect to the local K\"ahler metric of the l.c.K structure on $M,$ and $\tilde{\eta}$ the corresponding mean curvature vector. Then, (see \cite{DO}, Theorem 1.2 or \cite{Vaisman1} or \cite{Vaisman}), 
		$\tilde{A}(X,Y)=A(X,Y)+\frac{1}{2}g(X,Y)B^\perp,$ for all $X,Y\in \Gamma(TK),$ and $\tilde{\eta}=\eta+\frac{1}{2}B^\perp.$
		
		Since complex submanifolds of K\"ahler manifolds are minimal, $\tilde{\eta}=0,$ and consequently \begin{equation}\label{e3}
			\eta=-\frac{1}{2}B^\perp.
		\end{equation}
		Using (\ref{e1}), (\ref{e2}), (\ref{e3}), we obtain: $B^\perp=-\frac{1}{m}V^\perp,$ which implies that $V+mB$ is a vertical vector field.
		
		Conversely, if $V+mB$ is a vertical vector field, then $B^\perp=-\frac{1}{m}V^\perp.$ Since $\tilde{\eta}=0,$ and using (\ref{e3}), we get $\eta=\frac{1}{2m}V^\perp$ and, from (\ref{e1}), $\rm{trace}(A)-V^\perp=0.$ So $K$ is $V$-minimal.
	\end{proof}
	\begin{remark}
		Let $\varphi:M\rightarrow M^\prime$, $M$ compact, a l.c.K submersion, and $V$ a vector field on $M$. Denote by $\omega,$ $\omega^\prime$ the Lee forms of $M$ and $M^\prime$, and by $B$ and $B^\prime$ the corresponding Lee vector fields. From \cite{DO}, Proposition 10.1, $\omega(X)=\omega^\prime(X^\prime)\circ \varphi, $ for any horizontal vector field $X$ in $M$, which is $\varphi$ related to $X^\prime\in \Gamma(TM^\prime),$ i.e $d\varphi (X)=X^\prime,$ and $\mathcal{H}(B)$ is a horizontal vector field $\varphi$ related to $B^\prime,$ given by the decomposition $B=\mathcal{V}(B)+\mathcal{H}(B)$ of $B$ into vertical and horizontal component with respect to $M.$
		
		Il the fibers of $\varphi$ are $V$-minimal, then $V+mB$ is a vertical vector and $d\varphi(B)=-\frac{1}{m}d\varphi(V).$
		
		Then, $-\frac{1}{m}d\varphi(V)=d\varphi(\mathcal{V}(B)+\mathcal{H}(B))=d\varphi(\mathcal{H}(B))=B^\prime.$
		So the Lie vector field on $M^\prime$ is $-\frac{1}{m}d\varphi(V).$
	\end{remark}

	\section{$V$-Pseudo Harmonic Morphisms}
	\subsection{Pseudo Harmonic Morphisms and Pseudo Horizontally Homothetic Maps}
	
	The notion of harmonic morphisms can be generalized when the target manifold is endowed with a K\" ahler structure (see \cite{Loubeau}, \cite{C}).
	
	Let us consider a smooth map $\varphi :(M^m,g)\rightarrow (N^{2n},J,h)$ from a Riemannian manifold to a K\" ahler manifold.
	The map $\varphi$ is said to be a {\em pseudo-harmonic morphism} (shortening PHM)
	if and only if it pulls back local holomorphic functions on $N$ to local harmonic maps 
	from $M$ to $\bf C$.
	
	For any $x\in M,$ denote by $d\varphi ^*_x:T_{\varphi (x)}N\rightarrow T_xM$
	the adjoint map of the tangent linear map
	$d\varphi _x:T_xM\rightarrow T_{\varphi (x)}N.$ 
	
	If $X$ is a local section on the pull-back bundle $\varphi^{-1}TN$, then $d\varphi ^*(X)$ is a local horizontal vector field on $M$.
	
	\begin{definition}[see \cite{Loubeau}]
		The map $\varphi $ is called {\em 
			pseudo-horizontally (weakly) conformal (shortening PHWC) at $x\in M$}
		if $[d\varphi _x \circ d\varphi _x ^*,J]=0.$ 
		\medskip
		
		The map $\varphi $ is called {\em 
			pseudo-horizontally (weakly) conformal}
		if it is pseudo-horizontally (weakly) conformal  
		at every point of $M$. 
	\end{definition}
	
	Then, {\em pseudo-harmonic morphism} can also be characterised as harmonic, pseudo-horizontally (weakly) conformal maps (see \cite{Loubeau}, \cite{C}).
	
	\medskip
	
	The local description of PHWC condition is given
	by the following (see \cite{Loubeau}): let $(x_i)_{i=\overline{1,m}}$ be real local coordinates on $M$, $(z_\alpha)_{\alpha=\overline{1,n}}$ be complex local coordinates on $N$, and
	$\varphi ^\alpha=z_\alpha\circ \varphi, $ $\forall \alpha=1,...,n$.
	Then the PHWC condition for $\varphi $ reads: 
	\begin{equation}\label{phwc}
		\sum\limits_{i,j=1}^mg^{ij}_M
		\frac{\partial \varphi ^\alpha}{\partial x_i}
		\frac{\partial \varphi ^\beta}{\partial x_j}=0
	\end{equation}
	for all $\alpha,\beta=1,...,n$.
	
	\medskip
	
	A special class of pseudo-horizontally weakly conformal maps,
	are pseudo-horizontally homothetic maps.
	
	\begin{definition}[see \cite{AAB}]
		A map $\varphi :(M^m,g)\rightarrow (N^{2n},J,h)$ is called
		{\em pseudo-horizontally homothetic at $x$} (shortening PHH) if is
		PHWC at a point $x\in M$  and satisfy: 
		\begin{equation}
			d\varphi _x\left( (\nabla ^M_vd\varphi ^*(JY))_x\right) =
			J_{\varphi (x)}d\varphi _x\left( (\nabla ^M_vd\varphi ^*(Y))_x\right) ,
		\end{equation}
		for any horizontal tangent vector $v\in T_xM$ and
		any vector field $Y$ locally defined on a neighbourhood of $\varphi
		(x)$.
		
		A PHWC map
		$\varphi $ is called
		{\em pseudo-horizontally homothetic},  if and
		only if 
		\begin{equation}\label{phh}
			d\varphi (\nabla ^M_Xd\varphi ^*(JY))=
			Jd\varphi (\nabla ^M_Xd\varphi ^*(Y)),
		\end{equation}
		for any horizontal vector field $X$ on $M$ and
		any vector field $Y$ on $N$.
	\end{definition}
	
	The condition (\ref{phh}) is true for every horizontal
	vector field $X$ on $M$ and any section $Y$ of
	$\varphi ^{-1}TN$. So, instead of working with vector fields on $N,$ we work with the larger
	space of sections in the pull back bundle 
	$\varphi ^{-1}TN$.
	
	\medskip
	
	One of the basic properties of pseudo-horizontally homothetic maps (Proposition 3.3, \cite{AAB}) shows that a PHH submersion is a harmonic map if and only if it has minimal fibres. Also,
	pseudo-horizontally homothetic maps are good tools to construct minimal submanifolds (Theorem 4.1, \cite{AAB}).

	\subsection{$V$-Pseudo Harmonic Morphisms}
%
	
	Generalizing the class of harmonic maps and morphisms, respectively to $V$-harmonic maps and pseudo harmonic morphisms we obtain $V$-pseudo harmonic morphisms with a description similar to pseudo harmonic morphisms. 
	
	\begin{definition}\label{def:PVHM}
		Let $(M^m,g)$ be a Riemannian manifold of real dimension $m$, $(N^{2n},J,h)$ a Hermitian manifold of complex dimension $n$, $\varphi: M\rightarrow N$ a smooth map and $V$ a smooth vector field on $M$. The map $\varphi$ is called {\em $V$-pseudo harmonic morphism} (shortening $V$-PHM) if $\varphi$ is {\em $V$-harmonic} and {\it pseudo  horizontally weakly conformal.} 
	\end{definition}
	
	\medskip
	
	The characterization of {\it pseudo harmonic morphism} given in \cite{Loubeau} remain true in the general case of {\it $V$-harmonic} maps.
	
	\begin{theorem} \label{tpvhm}
		Let $\varphi: M\rightarrow N$ be a smooth map from a Riemannian manifold $(M^m,g)$ to a K\"ahler one $(N^{2n},J,h)$ and $V$ a smooth vector field on $M$. Then $\varphi$ is $V$-pseudo harmonic morphism ($V$-PHM) if and only if it pulls back local complex valued holomorphic functions on $N$ to local $V$-harmonic functions on $M$.
	\end{theorem}
	
	\begin{proof}
		We adapt the proof of Proposition 2 of \cite{Loubeau} to our context.
		Let us consider $p\in M$ a given point, $\{x_i\}_{i=\overline{1,m}}$ local coordinates at $p$, $\{z_\alpha\}_{\alpha=\overline{1,n}}$ local complex coordinates at $\varphi(p)\in N,$ and $\varphi^\alpha=z_\alpha\circ \varphi.$
		By ${^{M}\Gamma}_{ij}^k$ and ${^{N}L}_{\beta\gamma}^\alpha$ are denoted the Christoffel symbols on $M$ and $N$, respectively. The vector field $V$, in local coordinates in $M,$ reads: $V=\sum\limits_{i=1}^mV_i\frac{\partial}{\partial x_i}.$
		
		Suppose that $\varphi$ is {\it $V$-PHM} and $f: N \rightarrow {\bf C}$ is a local holomorphic function on $N$.
		
		Using the $V$-tension field of composition of two maps (Lemma \ref{comp}) and the $V$-harmonicity of the map $\varphi$ (Definition \ref{vharm}):
		$
		\tau_V(f\circ \varphi)=\mathrm{trace}\nabla df(d\varphi, d\varphi).
		$
		
		In order to prove that $\tau_V(f\circ \varphi)=0$, at point $p$, since $d\varphi(\frac{\partial}{\partial x_i})=\sum\limits_{\alpha=1}^n \frac{\partial \varphi^\alpha}{\partial x_i}\frac{\partial}{\partial z_\alpha},$ we have:
		
		\begin{equation}\label{eq1}
			\begin{array}{l}
				\nabla df\left(d\varphi\left(\frac{\partial}{\partial x_i}\right),d\varphi\left(\frac{\partial}{\partial x_j}\right)\right)=
		\sum\limits_{\alpha,\beta=1}^n\frac{\partial\varphi^\alpha}{\partial x_i}\frac{\partial \varphi^\beta}{\partial x_j}\nabla df\left(\frac{\partial}{\partial z_\alpha},\frac{\partial}{\partial z_\beta}\right).\\
			\end{array}
		\end{equation}
		
		To compute $\nabla df\left(\frac{\partial}{\partial z_\alpha},\frac{\partial}{\partial z_\beta}\right)$, first let us remark that
		$ f^{-1}T{\bf{C}}=N\times \bf{C}$ is the trivial vector bundle 
		and the fibre $(f^{-1}T{\bf{C}})_{\varphi(p)}=\bf{C}.$
		The induced connection on the pull-back bundle $f^{-1}T{\bf{C}}$ is defined by:
		$$
		\nabla^f_X\sigma:=X(\sigma), \forall \sigma\in \Gamma(f^{-1}T{\bf{C}}), (\sigma:N\rightarrow \bf{C} \mbox{ is a map} )
		$$
		So,
		\begin{equation}\label{eq2}
			\begin{array}{l}
				\nabla df\left(\frac{\partial}{\partial z_\alpha},\frac{\partial}{\partial z_\beta}\right)=
				\nabla^f_{\frac{\partial}{\partial z_\alpha}}df(\frac{\partial}{\partial z_\beta})-df(\nabla^N_{\frac{\partial}{\partial z_\alpha}}\frac{\partial}{\partial z_\beta})
				=\frac{\partial^2 f}{\partial z_\alpha \partial z_\beta}-\sum\limits_{\gamma=1}^n {^{N}L}_{\alpha\beta}^\gamma \frac{\partial f}{\partial z_\gamma}\\
			\end{array}
		\end{equation}
		
		From the equations (\ref{eq1}) and (\ref{eq2}) we obtain:
		$$
		\nabla df\left(d\varphi\left(\frac{\partial}{\partial x_i}\right),d\varphi\left(\frac{\partial}{\partial x_j}\right)\right)=
		\sum\limits_{\alpha,\beta=1}^n \frac{\partial\varphi^\alpha}{\partial x_i}\frac{\partial \varphi^\beta}{\partial x_j}\frac{\partial^2 f}{\partial z_\alpha \partial z_\beta}-\sum\limits_{\alpha,\beta,\gamma=1}^n \frac{\partial\varphi^\alpha}{\partial x_i}\frac{\partial \varphi^\beta}{\partial x_j}{^{N}L}_{\alpha\beta}^\gamma \frac{\partial f}{\partial z_\gamma}
		$$
		and
		\begin{equation}\label{etr}
			\begin{array}{c}
				\mathrm{trace}\nabla df\left(d\varphi\left(\frac{\partial}{\partial x_i}\right),d\varphi\left(\frac{\partial}{\partial x_j}\right)\right)=\\
				\\
				=\sum\limits_{\alpha,\beta=1}^n\left(\sum\limits_{i,j=1}^m g^{ij}\frac{\partial\varphi^\alpha}{\partial x_i}\frac{\partial \varphi^\beta}{\partial x_j}\right)\frac{\partial^2f}{\partial z_\alpha \partial z_\beta}-
				\sum\limits_{\alpha,\beta,\gamma=1}^n \left(\sum\limits_{i,j=1}^m g^{ij}\frac{\partial\varphi^\alpha}{\partial x_i}\frac{\partial \varphi^\beta}{\partial x_j}\right){^{N}L}_{\alpha\beta}^\gamma \frac{\partial f}{\partial z_\gamma}\\
			\end{array}
		\end{equation}
		As $\varphi$ is PHWC (see (\ref{phwc})), the term $\sum\limits_{i,j=1}^m g^{ij}\frac{\partial\varphi^s}{\partial x_i}\frac{\partial \varphi^k}{\partial x_j},$ in relation (\ref{etr}), vanishes, and hence
		$$\mathrm{trace}\nabla df\left(d\varphi\left(\frac{\partial}{\partial x_i}\right),d\varphi\left(\frac{\partial}{\partial x_j}\right)\right)=0.$$
		
		Conversely, consider $f:N\rightarrow \bf{C}$ a local complex holomorphic function, $V$ a smooth vector field on $M$, and $\varphi:M\rightarrow N$ a smooth map, such that $\tau_V(f\circ \varphi)=0$. 
		
		Applying the chain rule for $V$-harmonic maps (Lemma \ref{comp}), we get:
		\begin{equation}\label{eq5}
			0=df(\tau_V(\varphi))+\mathrm{trace} \nabla df(d\varphi,d\varphi)
		\end{equation}
		We compute the above equality in local coordinates, using the local description of the $V$-tension field (\ref{tauv}) and the equation (\ref{etr}): 
		\begin{equation}\label{eq3}
			\begin{array}{c}
				0=\sum\limits_{\alpha=1}^n\tau(\varphi)^\alpha\frac{\partial f}{\partial z_\alpha}+
				\sum\limits_{\alpha=1}^n\sum\limits_{i=1}^mV_i\frac{\partial\varphi^\alpha}{\partial x_i}\frac{\partial f}{\partial z_\alpha}\\
				\\
				+\sum\limits_{\alpha,\beta=1}^n\left(\sum\limits_{i,j=1}^m g^{ij}\frac{\partial\varphi^\alpha}{\partial x_i}\frac{\partial \varphi^\beta}{\partial x_j}\right)\frac{\partial^2f}{\partial z_\alpha \partial z_\beta}-
				\sum\limits_{\alpha,\beta,\gamma=1}^n \left(\sum\limits_{i,j=1}^m g^{ij}\frac{\partial\varphi^\alpha}{\partial x_i}\frac{\partial \varphi^\beta}{\partial x_j}\right){^{N}L}_{\alpha\beta}^\gamma \frac{\partial f}{\partial z_\gamma}\\
			\end{array}
		\end{equation}
		
		Choosing in (\ref{eq3}) particular holomorphic functions $f$, for example $f(z)=z_k, \forall k=\overline{1,n},$ and normal coordinates centred at $p$ (respectively, at $\varphi(p)$), then $\frac{\partial^2f}{\partial z_\alpha \partial z_\beta}=0,$ $\frac{\partial f}{\partial z_\gamma}=\delta_{k\gamma},$ and the Christoffel symbols of $N$ vanish. Then, 
		
		$\tau_V(f\circ\varphi)=0$ is equivalent to:
		$$\tau(\varphi)^k+\sum\limits_{i=1}^mV_i\frac{\partial \varphi^k}{\partial x_i}=0, \forall k=\overline{1,n},
		$$
		where the term $\tau(\varphi)^k+\sum\limits_{i=1}^mV_i\frac{\partial \varphi^k}{\partial x_i}=0 $ is the $k$ component of the $\tau_V(\varphi).$
		
		This proves that $\varphi$ is a $V$-harmonic map. It follows that the equation (\ref{eq5}) reduces to:
		$\mathrm{trace}\nabla df(d\varphi,d\varphi)=0.$
		
		In this last equality, choosing $f(z)=z_\alpha z_\beta, \forall \alpha, \beta=\overline{1,n}$ and normal coordinates in $M$ and $N$ respectively, $\frac{\partial^2 (z_\alpha z_\beta)}{\partial z_\alpha \partial z_\beta}=1,$ and ${^{N}L}_{\alpha\beta}^\gamma=0,$ implies 
		$$\sum\limits_{i,j=1}^m g^{ij}\frac{\partial\varphi^\alpha}{\partial x_i}\frac{\partial \varphi^\beta}{\partial x_j}=0
		$$ which means that $\varphi$ is PHWC.
	\end{proof}
	
	The following result gives a relation between $V$-harmonic morphisms, pseudo harmonic morphisms and $V$-pseudo harmonic morphisms, compare to Proposition 3 of \cite{Loubeau}.
	\begin{proposition}
		Let $(M,g)$ and $(N,h)$ be two Riemannian manifolds, $V$ a smooth vector field on $M$, $(P,J,p)$ a K\" ahler manifold and $\psi:M\rightarrow N$ and $\varphi:N\rightarrow P$ two smooth maps. If $\psi$ is $V$-harmonic morphism and $\varphi$ is pseudo harmonic morphism (PHM), then $\varphi \circ \psi$ is $V$-pseudo harmonic morphism ($V$-PHM).
	\end{proposition}
	
	\begin{proof}
		As $\varphi$ is a pseudo harmonic morphism (see Section 1.2), for any $f:W\rightarrow \mathbf{C}$ a local complex valued holomorphic function defined in an open subset $W\subset P,$ with $\varphi^{-1}(W)\subset N$ non-empty, $f\circ \varphi:\varphi^{-1}(W)\rightarrow \mathbf{C}$ is a local harmonic function on $N$.
		
		Using Corollary \ref{lift} for the $V$-harmonic morphism $\psi$ and the local harmonic function  $f\circ \varphi:\varphi^{-1}(W)\rightarrow \mathbf{C}$, on $N$, we get $f\circ\varphi\circ\psi:\psi^{-1}(\varphi^{-1}(W))\rightarrow \mathbf {C},$ a local $V$-harmonic function on $M$.
		
		So, $\varphi\circ \psi:M\rightarrow P$ pulls-back local holomorphic functions on $P$ to local $V$-harmonic functions on $M$. By Theorem \ref{tpvhm}, $\varphi\circ \psi$ is $V$-PHM. 
	\end{proof}

	\begin{theorem}\label{p13}
		Let $\varphi: (M^m,g)\rightarrow (N^{2n},J,h), n\geq 2,$ be a pseudo horizontally homothetic submersion (PHH). Then $\varphi$ is $V$-harmonic if and only if $\varphi$ has $V$-minimal fibres.
	\end{theorem}
	
	\begin{proof}
		
		Similar to Proposition 3.3, \cite{AAB}, we can choose in the pull-back bundle $\varphi^{-1}TN$ a local frame $\{e_1, e_2,\dots,e_n,Je_1,Je_2,\dots,Je_n\}$ such that 
		$$\{d\varphi^*(e_1),\dots,d\varphi^*(e_n),d\varphi^*(Je_1),\dots,d\varphi^*(Je_n)\}$$ is an orthogonal frame in the horizontal distribution: choose $e_1\in \Gamma(\varphi^{-1}TN)$ a non-vanishing section and, using the PHWC property of $\varphi$ and the fact that $d\varphi^*(e_1)$ is an horizontal vector, we get the orthogonality  of $d\varphi^*(e_1)$ and $d\varphi^*(Je_1)$. At step $k$, take $e_k$ orthogonal on both vectors $d\varphi\circ d\varphi^*(e_i)$ and $d\varphi\circ d\varphi^*(Je_i)$, $\forall i\leq k-1.$

		If we denote by $E_k=d\varphi^*(e_k), E_k^\prime=d\varphi^*(Je_k),$ then we have an orthogonal frame in the horizontal distribution $\mathcal{H}(TM),$ $\{E_1,E_2,\dots,E_n,E_1^\prime,\dots,E_n^\prime\}.$ Note that, by the PHWC condition of $\varphi$, 
		\begin{equation}\label{7a}
			g(E_i,E_i)=g(E_i^\prime,E_i^\prime).
		\end{equation}
		
		It was proved in \cite{AAB} that, the property of the induced connection $\nabla^\varphi$ on $\varphi^{-1}TN$ (Lemma 1.16, \cite{Urakawa}) and the PHH condition, imply:
		
		\begin{equation}\label{fi}
			\nabla_{E_i^\prime}^\varphi d\varphi(E_i^\prime)=-\nabla_{E_i}^\varphi d\varphi(E_i)+J d\varphi([E_i^\prime,E_i])
		\end{equation}
		and 
		\begin{equation}\label{fia}
			\frac{1}{g(E_i,E_i)}Jd\varphi([E_i^\prime,E_i])=\frac{1}{g(E_i^\prime,E_i^\prime)}\nabla_{E_i^\prime}^\varphi d\varphi(E_i^\prime)+\frac{1}{g(E_i,E_i)}\nabla_{E_i}^\varphi d\varphi(E_i)\\
		\end{equation}
		
		\begin{equation}
			\label{9a}
			d\varphi(\nabla^M_{E_i^\prime}E_i^\prime)= -d\varphi(\nabla^M_{E_i}E_i)+Jd\varphi([E_i^\prime,E_i]).
		\end{equation}
		
		and
		
		\begin{equation}\label{9}
			\frac{1}{g_M(E_i,E_i)}Jd\varphi([E_i^\prime,E_i])=
			\frac{1}{g_M(E_i,E_i)}d\varphi(\nabla^M_{E_i}E_i)+\frac{1}{g_M(E_i^\prime,E_i^\prime)}d\varphi(\nabla^M_{E_i^\prime}E_i^\prime).
		\end{equation}
		
		We choose $\{u_1,u_2,\dots,u_s\}$ an orthonormal basis for the vertical distribution $\mathcal{V}(TM).$ With this notation and using (\ref{fi}), (\ref{fia}), (\ref{9a}), (\ref{9}), it was shown in Proposition 3.3. in \cite{AAB}, that 
		\begin{equation}
			\label{eq:tau-PHH}
			\tau(\varphi)=-d\varphi(\sum\limits_{j=1}^s\nabla^M_{u_j}u_j).
		\end{equation}

			If $\varphi$ is a $V$-harmonic map, then $\tau_V(\varphi)=0$ and, from (\ref{eq:tau-PHH}) we obtain:
			$$0=-d\varphi(\sum\limits_{j=1}^s\nabla^M_{u_j}u_j)+d\varphi(V),
			$$
			which is equivalent with the $V$-minimality of the fibres.
			
			Conversely, suppose $\varphi$ has $V$-minimal fibres.
			As $\varphi$ is a PHH submersions, in the above constructed frame $\{u_1,\dots,u_s,E_1,\dots,E_n,E^\prime_1,\dots,E^\prime_n\},$ the $V$-tension field of $\varphi$ reads:
			$$
			\tau_V(\varphi)=-d\varphi(\sum\limits_{j=1}^s\nabla^M_{u_j}u_j)+d\varphi(V)
			$$
			The $V$-minimality of the fibres imply $d\varphi(\mathrm{trace}(A)-V)=0$.
			
			Let us choose $y\in N,$ and denote the fibre by $K=\varphi^{-1}(y).$
			
			Using (\ref{7a}), we have
			$$
			\begin{array}{c}
				\mathrm{trace} (A)-V=\sum\limits_{i=1}^sA(u_i,u_i)-V
				=\frac{1}{g_M(u_i,u_i)}\sum\limits_{i=1}^s(\nabla^M_{u_i}u_i-\nabla^K_{u_i}u_i)-V\\
			\end{array}
			$$
			Since $\nabla^K_{u_i}u_i$ is vertical, $d\varphi(\nabla^K_{u_i}u_i)=0,$ and
			$$d\varphi(\mathrm{trace}(A)-V)=\sum\limits_{i=1}^s d\varphi(\nabla^M_{u_i}u_i)-d\varphi(V),
			$$
			and hence $\varphi$ is $V$-harmonic.
		\end{proof}
		

		The construction of minimal submanifolds was done for horizontally homothetic harmonic morphisms (see \cite{BG}) and generalised for the case of pseudo-horizontally homothetic harmonic submersions (see \cite{AAB}). Replacing harmonicity by $V$-harmonicity, a similar result can be proved.  
		
		\begin{theorem}\label{thm:pullback}
			Let $(M^m,g)$ be a Riemannian manifold, $(N^{2n},J,h)$ be a K\"ahler manifold, $V$ a smooth vector field on $M$ and $\varphi: M\rightarrow N$ be a pseudo-horizontally homothetic (PHH), $V$-harmonic submersion.
			
			If $P^{2p}\subset N^{2n}$ is a complex submanifold of $N$, then $\varphi^{-1}(P)\subset M$ is a $V$-minimal submanifold of $M$.
		\end{theorem}
		
		\begin{proof}
			We shall use an orthogonal local basis in the tangent bundle of $\varphi^{-1}(P),$ as in (\cite{AAB}). Specifically, let us consider the decomposition of $TM$ into the vertical distribution $V$ and the horizontal one $H$, $T_xM=V_x\oplus H_x,$ for any point $x\in M.$ Denote $\varphi^{-1}(P),$ by $K,$ $H_1=TK\cap H$ and, for any $x\in K,$ $H_x=H_{1_x}\oplus H_{2_x}.$ Also, $T_xK=V_x\oplus H_{1_x},$ $H_{1_x}=\{u\in H_x, d\varphi_x(u)\in T_{\varphi(x)}P\},$ $\forall x\in K.$
			
			We choose a linearly independent system of local sections 
			$\{e_1,\dots,e_p,Je_1,\dots,Je_p \}$ in $\varphi^{-1}TN,$  such that, if we denote by $E_i=d\varphi^*(e_i)$ and $E_i^\prime=d\varphi^*(Je_i),$ the restriction of the system to $K$ is a local orthonormal basis in $H_1.$
			
			The system was chosen in the following way. Since $d\varphi_{x_{|_{H_x}}}\circ d\varphi^*_x $ is a linear isomorphism, take $v_1$ a local section of $\varphi^{-1}TP,$ non-vanishing along $K$ and choose the section $e_1$ of $\varphi^{-1}TN,$ such that $(d\varphi_{|H}\circ d\varphi^*)(e_1)=v_1$. Using the PHWC condition of $\varphi$, $d\varphi(d\varphi^*(Je_1))$ is also a section of $\varphi^{-1}TP$ and $d\varphi^*(e_1)$ and $d\varphi^*(Je_1)$ are orthogonal local vector fields in $H$ and also in $H_1$. By induction, at step $k$, we take  $e_k$ perpendicular on both $(d\varphi\circ d\varphi^*)(e_i)$ and $(d\varphi\circ d\varphi^*)(Je_i), \forall 1\leq i\leq k-1,$ such that $(d\varphi\circ d\varphi^*)(e_i)$ is a section in $\varphi^{-1}TP.$
			
			In the sequel, the computations are  performed on points of $K$ only.
			
			Consider $\{u_1,u_2,\dots,u_r\}$ an orthonormal local basis in the vertical distribution $V$ such that $\{u_1,u_2,\dots,u_r,E_1,\dots,E_p,E^\prime_1,\dots,E^\prime_p\}$ is an orthogonal local basis in $TK.$ Denote by $A$ the second fundamental form of the submanifold $K$.
			
			The $V$-minimality of $K$ is equivalent to $\mathrm{trace}(A)-V\in \Gamma(TK),$ which reads:
			\smallskip
			
			$\begin{array}{c}
				\sum\limits_{i=1}^r\left[\frac{1}{g_M(u_i,u_i)}d\varphi(\nabla^M_{u_i}u_i-\nabla^K_{u_i}u_i)\right]
				+\sum\limits_{i=1}^p\left[\frac{1}{g_M(E_i,E_i)}d\varphi(\nabla^M_{E_i}E_i-\nabla^K_{E_i}E_i)\right]\\\\
				+\sum\limits_{i=1}^p\left[\frac{1}{g_M(E^\prime_i,E^\prime_i)}d\varphi(\nabla^M_{E^\prime_i}E^\prime_i-\nabla^K_{E^\prime_i}E^\prime_i)\right]-d\varphi(V) \mbox{ is a section of } \varphi^{-1}TP.
			\end{array}$
			
			\smallskip
			Using relations (\ref{7a}) and (\ref{9a}), we compute:
			\smallskip
			
				$\begin{array}{l}
					\sum\limits_{i=1}^p\left[\frac{1}{g_M(u_i,u_i)}d\varphi(\nabla^M_{u_i}u_i-\nabla^K_{u_i}u_i)\right]
					+\sum\limits_{i=1}^p\left[\frac{1}{g_M(E_i,E_i)}d\varphi(\nabla^M_{E_i}E_i)+\frac{1}{g_M(E^\prime_i,E^\prime_i)}d\varphi(\nabla^M_{E^\prime_i}E^\prime_i)\right]\\
					\\
					-\sum\limits_{i=1}^p\left[\frac{1}{g_M(E_i,E_i)}(\nabla^K_{E_i}E_i)+\frac{1}{g_M(E^\prime_i,E^\prime_i)}(\nabla^K_{E^\prime_i}E^\prime_i)\right]-d\varphi(V)\\
					\\
					=\sum\limits_{i=1}^pd\varphi(\nabla^M_{u_i}u_i)-\sum\limits_{i=1}^pd\varphi(\nabla^K_{u_i}u_i)-d\varphi(V)
				\end{array}$

			As $\varphi$ is a pseudo-horizontally homothetic $V$-harmonic submersion, from Theorem (\ref{p13}), $\varphi$ has $V$-minimal fibres, which is equivalent to: $d\varphi(\sum\limits_{i=1}^p\nabla^M_{u_i}u_i)=d\varphi(V).$
			
			Hence, $\mathrm{trace}(A)-V\in \Gamma(TK)$ is equivalent to $\sum\limits_{i=1}^pd\varphi(\nabla^K_{u_i}u_i)$ is a section of $\varphi^{-1}TP.$ 
		\end{proof}

		\noindent
		{\bf Acknowledgements.}
		This work was partially supported by Romanian Ministry of Education,
		Program PN-III, Project number PN-III-P4-ID-PCE-2020-0025, Contract 30/04.02.2021 and by "Dunarea de Jos" University of Galati Grant, Project number  RF 2488/31.05.2024.

		\vskip .2cm
		
		\small \noindent Department of Mathematics and Computer Science, Faculty of Sciences and Environmental, \\
		University "Dun\u area de Jos" of Gala\c ti,
		111 \c Stiin\c tei Str., RO-800189, Gala\c ti, Romania.\\\ 
		{\it Email address}: Monica.Aprodu@ugal.ro 
		\\\\

\begin{thebibliography}{Ehr}
			\bibitem{AAB} Aprodu Monica Alice, Aprodu Marian, Br\^ anz\u anescu Vasile, {\it A Class of Harmonic Submersions and Minimal Submanifolds}, International Journal of Mathematics, Vol. 11, No. 9, (2000), 1177-1191. 
			\bibitem{BE} Baird Paul, Eells  James, {\it A conservation law for harmonic maps}, Geometry Symposium Utrecht 1980, Lecture Notes in Mathematics, 894 Springer-Verlag (1981), 1-25.
			\bibitem{BG} Baird Paul, Gudmundsson Sigmundur, {\it $p$-harmonic maps and minimal submanifolds}, Math. Ann, No. 294 (1992), 611 - 624.
			\bibitem{BairdWood} Baird Paul, Wood John C., {\it Harmonic Morphisms Between Riemannian Manifolds}, Clarendon Press-Oxford (2003), ISBN 0-19-850362-8.
			\bibitem{CJW} Chen Qun, Jost J\" urgen, Wang Guofang, {\it A Maximum Principle for Generalizations of Harmonic Maps in Hermitian, Affine, Weyl, and Finsler Geometry}, J. Geom. Anal., No. 25 (2015), 2407–2426.
			\bibitem{CJQ} Chen Qun, Jost J\" urgen, Qiu Hongbing, {\it Existence and Liouville theorems for V-harmonic maps from complete manifolds}, Ann. Global Anal. Geom.  Vol. 42 (2012), No. 4, 565-584.
			\bibitem{CQ} Chen Qun, Qiu Hongbing, {\it Rigidity of self-shrinkers and translating solitons of mean curvature flows}, Adv. Math. Vol. 294, (2016), 517-531.
			\bibitem{C} Chen Jingyi, {\it Structures of Certain Harmonic Maps into Kähler Manifolds}, International Journal of Mathematics, Vol. 8, No. 5, (1997), 573-581.
			\bibitem{DO} Dragomir Sorin, Ornea Liviu, {\it Locally Conformal Kähler Geometry}, Progress in Mathematics 155,  Birkhäuser Boston (1998), ISBN 978-1-4612-7387-5.
			\bibitem{EL} Eells  James, Lemaire Luc, {\it A report on harmonic maps}, Bulletin of the London Mathematical Society 10(1) (1978), 1 - 68.
			\bibitem{HL} Hsiang Wu-Yi, Lawson H.Blaine, {\it Minimal submanifolds of low cohomogeneity},  J. Differential Geom. No. 5(1-2) (1971), 1-38.
			\bibitem{JL} Lee John M., {\it Introduction to Smooth Manifolds},  Springer, Graduate Texts in Mathematics 218 (2013), ISBN 978-1-4419-9981-8.
			\bibitem{Loubeau} Loubeau Eric, {\it Pseudo Harmonic Morphisms}, International Journal of Mathematics, Vol. 8, No. 7, (1997), 943-957. 
			\bibitem{Marrero} Marrero Juan Carlos, Rocha Juan, {\it Locally Conformal K\" ahler Submersions}, Geometriae Dedicata, Vol. 52, (1994), 271-289.
			\bibitem{Neill} O'Neill Barret, {\it The fundamental equation of a submersion}, Michigan Math. J. Vol. 13, No. 4, (1966), 459-469.
			\bibitem{Q} Qiu Hongbing, {\it The heat flow of V-harmonic maps from complete manifolds into regular balls}, Proc. Amer. Math. Soc. Vol. 145, (2017), No. 5, 2271-2280.
			\bibitem{Urakawa} Urakawa Hajime, {\it Calculus of Variations and Harmonic Maps}, Translations of Mathematical Monographs 132 AMS (1993), ISBN 0-8218-4581-0.
			\bibitem{Vaisman} Vaisman Izu, {\it Generalized Hopf Manifolds}, Geometriae Dedicata, Vol. 13, (1982), 231-255.
			\bibitem{Vaisman1} Vaisman Izu, {\it On locally conformal almost K\"ahler manifolds}, Israel J.Math, Vol. 24, (1976), 338-351.
			\bibitem{Zhao} Zhao Guangwen, {\it V-Harmonic Morphisms Between Riemannian Manifolds}, Proc. Amer. Math. Soc., Vol.148, No. 3, (2020), 1351-1361. 
			\bibitem{White} White Brian, {\it The Space of Minimal Submanifolds for Varying Riemannian Metrics}, Indiana University Mathematics Journal, Vol. 40, No. 1 (1991), 161 - 200.
			\bibitem{Wood} Wood John C., {\it Harmonic morphisms, foliations and Gauss map}, Contemp. Math.49 (1986), 145-184.
			
		\end{thebibliography}
	\end{document}